\documentclass{amsart}
\usepackage{amscd,amssymb}
\usepackage[all]{xy}
\begin{document}

\makeatletter
\def\revddots{\mathinner{\mkern1mu\raise\p@
  \vbox{\kern7\p@\hbox{.}}\mkern2mu
  \raise4\p@\hbox{.}\mkern2mu\raise7\p@\hbox{.}\mkern1mu}}
\makeatother
\newcommand{\Qed}{\hfill \mbox{$\blacksquare$}}

\newcommand{\lra}{\longrightarrow}
\newcommand{\lla}{\longleftarrow}
\newcommand{\Ra}{\Rightarrow}
\newcommand{\da}{\downarrow}
\newcommand{\ra}{\rightarrow}
\newcommand{\ovr}{\overrightarrow}

\newcommand{\kk}{\mathbb{K}}
\newcommand{\zz}{\mathbb{Z}}
\newcommand{\nn}{\mathbb{N}}

\renewcommand{\Im}{\operatorname{Im}\nolimits}
\newcommand{\Ker}{\operatorname{Ker}\nolimits}
\newcommand{\car}{\operatorname{char}\nolimits}
\newcommand{\Hom}{\operatorname{Hom}\nolimits}
\newcommand{\Ext}{\operatorname{Ext}\nolimits}
\newcommand{\HH}{\operatorname{HH}\nolimits}
\newcommand{\opp}{\operatorname{op}\nolimits}
\newcommand{\gr}{\operatorname{gr}\nolimits}
\newcommand{\ev}{\operatorname{ev}\nolimits}

\newcommand{\ot}{\otimes}
\newcommand{\lan}{\Lambda_N}
\renewcommand{\L}{\Lambda}
\newcommand{\oa}{\bar{a}}
\newcommand{\mo}{\mathfrak{o}}
\newcommand{\mt}{\mathfrak{t}}

\newcommand{\rrad}{\mathfrak{r}}
\newcommand{\rad}{\operatorname{rad}\nolimits}
\newcommand{\m}{\frak{m}}

\newcommand{\N}{\mathcal{N}}
\newcommand{\X}{\mathcal{X}}
\newcommand{\Y}{\mathcal{Y}}
\newcommand{\Z}{\mathbb{Z}}
\newcommand{\E}{\mathcal{E}}
\newcommand{\F}{\mathcal{F}}
\newcommand{\G}{\mathcal{G}}
\newcommand{\B}{\mathcal{B}}
\newcommand{\C}{\mathcal{C}}
\newcommand{\D}{\mathcal{D}}
\newcommand{\U}{\mathcal{U}}
\newcommand{\Pp}{\mathcal{P}}
\newcommand{\R}{\mathcal{R}}
\newcommand{\I}{\mathcal{I}}
\newcommand{\M}{\mathcal{M}}
\newcommand{\T}{\mathcal{T}}
\newcommand{\Q}{\mathcal{Q}}
\newcommand{\h}{\mathcal{H}}
\newcommand{\A}{\mathcal{A}}

\newtheorem{lem}{Lemma}[section]
\newtheorem{prop}[lem]{Proposition}
\newtheorem{cor}[lem]{Corollary}
\newtheorem{thm}[lem]{Theorem}
\newtheorem{bit}[lem]{}
\theoremstyle{definition}
\newtheorem{defin}[lem]{Definition}
\newtheorem*{remark}{Remark}
\newtheorem{example}[lem]{Example}

\setlength{\textheight}{21cm}

\title[The second Hochschild cohomology group \ldots]
{The second Hochschild cohomology group for one-parametric self-injective algebras}
\author[Al-Kadi]{Deena Al-Kadi}
\address{Deena Al-Kadi\\Department of Mathematics\\
Taif University\\ Saudi Arabia}
\email{dak12le@hotmail.co.uk}

\begin{abstract}
In this paper we determine the second Hochschild cohomology group for a class of self-injective algebras of tame
representation type namely, those which are standard one-parametric but not weakly symmetric. These were classified up to derived equivalence by Bocian, Holm and Skowro\'nski in \cite{BHS}. We connect this to the deformation of these algebras.
\end{abstract}

\date{\today}
\maketitle

\section*{Introduction}
This paper determines the second Hochschild cohomology group for all standard one-parametric but not weakly symmetric self-injective algebras of tame representation type. Bocian, Holm and Skowro\'nski give, in \cite{BHS}, a classification of these algebras by quiver and relations up to derived equivalence. The algebras in \cite{BHS} are divided into two types, namely the algebra $\L = \L(p,q,k,s,\lambda)$ where $p,q,s,k$
are integers such that $p,q \geq 0, k \geq 2, 1 \leq s \leq k-1,\gcd(s,k) = 1, \gcd(s+2,k) = 1$ and $\lambda \in K\setminus\{0\}$ and the algebra $\L = \Gamma^*(n)$ where $n \geq 1$. Thus the second Hochschild cohomology group will be known for all the classes of the algebras given in \cite{BHS}. We remark that an algebra of the type $\L(p,q,k,s,\lambda)$ is never isomorphic to an algebra of the type $\Gamma^*(n)$ as their stable Auslander-Reiten quivers are not isomorphic. We refer the reader to \cite{BHS} which gives precise conditions for two algebras of the same type $\L(p,q,k,s,\lambda)$ or $\Gamma^*(n)$ to be isomorphic.

We start, in Section \ref{sec1}, by introducing the algebras $\L$, for both types, by quiver and relations. Section \ref{sec2} of this paper describes the projective resolution of \cite{GS} which we use to find ${\HH^2(\L)}$. In the third section, we determine ${\HH^2(\L)}$ for the algebra $\L = \L(p,q,k,s,\lambda)$, considering separately the cases $1 \leq s \leq k-2$ and $s = k-1$. The main result in this section is Theorem \ref{thm2}, which shows that ${\HH^2(\L)}$ has dimension 1 for $1 \leq s \leq k-1$. This group measures the infinitesimal deformations of the algebra $\L$; that is, if ${\HH}^2(\L) = 0$ then $\L$ has no non-trivial deformations, which is not the case here. We include, in Section \ref{sec3}, Theorem \ref{thm3} where we find a non-trivial deformation $\L_\eta$ of $\L$ associated to our non-zero element $\eta$ in ${\HH^2(\L)}$. This illustrates the connection between the second Hochschild cohomology group and deformation theory. In the final section, we determine ${\HH^2(\L)}$ for $\L = \Gamma^*(n)$. The main result in Section \ref{sec4} is Theorem \ref{theorem1} which shows that $\dim\,{\HH^2(\L)} = 2.$ The results we found in this paper are in contrast to the majority of self-injective algebras of finite
representation type (see \cite{D}). Since Hochschild cohomology is invariant under derived equivalence, the second Hochschild cohomology group is now known for the standard one-parametric but not weakly symmetric self-injective algebras of tame representation type which are derived equivalent to the algebra of the type $\L(p,q,k,s,\lambda)$ or $\Gamma^*(n)$.

\section{The one-parametric self-injective algebras}\label{sec1}
In this chapter we describe the algebras of \cite{BHS}. We start with the algebra $\L = \L(p,q,k,s,\lambda)$. Let $K$ be an
algebraically closed field and let $p,q,s,k$
be integers such that $p,q \geq 0, k \geq 2, 1 \leq s \leq k-1,\gcd(s,k) =
1, \gcd(s+2,k) = 1$ and $\lambda \in K\setminus\{0\}$. From \cite[Section 5]{BHS}, $\L(p,q,k,s,\lambda)$
has quiver ${\mathcal Q}(p,q,k,s)$:

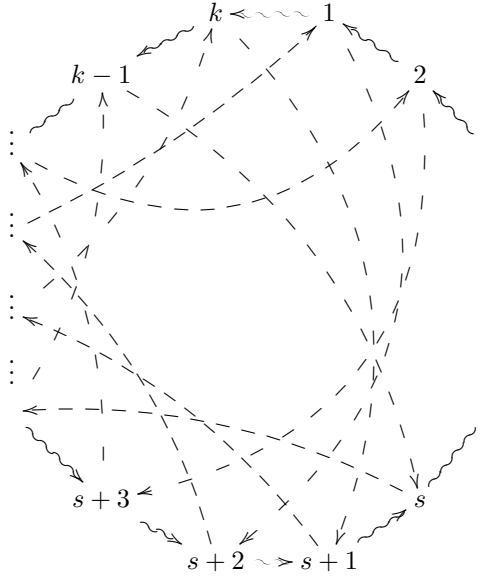
$$\xymatrix@R=10pt@C=15pt{
& & k\ar@{~>}[dl]\ar@{-->}@/^3pc/[dddddddrrrll] & 1\ar@{~>}[l]\ar@{-->}@/^4pc/[dddddddlllrr] & & \\
& {k-1}\ar@{~}[dl]\ar@{-->}@/^2pc/[dddddrrr] &  && 2\ar@{~>}[ul]\ar@{-->}@/^4pc/[lllddddd] & \\
\vdots\ar@{-->}@/_3pc/[urrrr] & & && & \vdots \ar@{~>}[ul]\\
\vdots \ar@{-->}@/_/[uuurrr]& & & && \vdots \\
\vdots\ar@{-->}@/_1pc/[uuuurrrl] & & && & \vdots \\
\raisebox{3ex}{\vdots} \ar@{~>}[dr]\ar@{-->}@/_1pc/[uuuur] & & & && \raisebox{3ex}{\vdots}\\
& s+3\ar@{~>}[dr]\ar@{-->}@/_1pc/[uuuul] & &&  s\ar@{~}[ur]\ar@{-->}@/_1pc/[ullll] & \\
& & s+2\ar@{~>}[r]\ar@{-->}@/_1pc/[uuuull] & s+1\ar@{~>}[ur]\ar@{-->}@/_1pc/[uuulll] & & \\
}$$

\vspace{1cm}

where, for any $i \in \{1,2, \ldots, k\}$,
$\xymatrix{
i  \ar@{~>}[rr]& &i-1
}$ denotes the path\\
$$\xymatrix{
\ar[r]^{\alpha_{(i,0)}}i & \ar[r]^{\alpha_{(i,1)}}   (i,1) &  (i,2) \ar[r]^{\alpha_{(i,2)}} & \cdots  \ar[r]^{\alpha_{(i,q-1)}}  & \ar[r]^{\alpha_{(i,q)}} (i,q) & i-1,\\
}$$

and $\xymatrix{
i-1\ar@{-->}[rr] && i+s
}$ denotes the path
$$\xymatrix{
\ar[r]^{\beta_{i^0}} i-1 & \ar[r]^{\beta_{i^1}} i^1 & \ar[r]^{\beta_{i^2}}i^2 & \cdots \ar[r]^{\beta_{i^{p-1}}} & \ar[r]^{\beta_{i^p}} i^p & i+s.
}$$\\

Then $\Lambda = K{\mathcal Q}(p,q,k,s) / I(p,q,k,s,\lambda)$ where $I(p,q,k,s,\lambda)$ is the ideal generated by the relations

$\bullet \beta_{i^p}\beta_{(s+i+1)^0}$, \hspace{1cm} for $i = 1, 2, \ldots, k$,

$\bullet \alpha_{(i,q)} \alpha_{(i-1,0)}$, \hspace{.9cm} for $i = 1, 2, \ldots, k$,

$\bullet \alpha_{(i,t')}\alpha_{(i,t'+1)} \cdots \alpha_{(i,q)} \beta_{i^0} \beta_{i^1} \cdots \beta_{i^p} \alpha_{(s+i,0)} \alpha_{(s+i,1)} \cdots \alpha_{(s+i,t')},$

\hspace{3.2cm} for $t' = 0, 1, \ldots, q$, $i = 1, 2, \ldots, k$,

$\bullet \beta_{i^j}\beta_{i^{j+1}} \cdots \beta_{i^p} \alpha_{(s+i,0)} \alpha_{(s+i,1)} \cdots \alpha_{(s+i,q)} \beta_{(s+i)^0}\beta_{(s+i)^1} \cdots \beta_{(s+i)^j},$

\hspace{3.2cm} for $j = 0, 1, \ldots, p$, $i = 1, 2, \ldots, k$,

$\bullet \alpha_{(i,0)} \alpha_{(i,1)} \cdots \alpha_{(i,q)} \beta_{i^0} \beta_{i^1} \cdots \beta_{i^p}$

\hspace{2cm}$- \beta_{(i+1)^0} \beta_{(i+1)^1} \cdots \beta_{(i+1)^p} \alpha_{(s+i+1,0)} \alpha_{(s+i+1,1)}\cdots \alpha_{(s+i+1,q)},$

\hspace{3.3cm}for $i = 2, \ldots, k$, and

$\bullet \alpha_{(1,0)} \alpha_{(1,1)} \cdots \alpha_{(1,q)} \beta_{1^0} \beta_{1^1} \cdots \beta_{1^p} - \lambda \beta_{2^0} \beta_{2^1} \cdots \beta_{2^p} \alpha_{(s+2,0)} \alpha_{(s+2,1)} \cdots \alpha_{(s+2,q)},$

\hspace{3.2cm} where $\lambda \in K\setminus\{0\}$.

\vspace{.5cm}

Next we describe the algebra $\L = \Gamma^*(n).$ For $n \geq 1$, $\Gamma^*(n)$ is given in \cite[Section 6]{BHS} by the quiver ${\mathcal Q}(n)$:

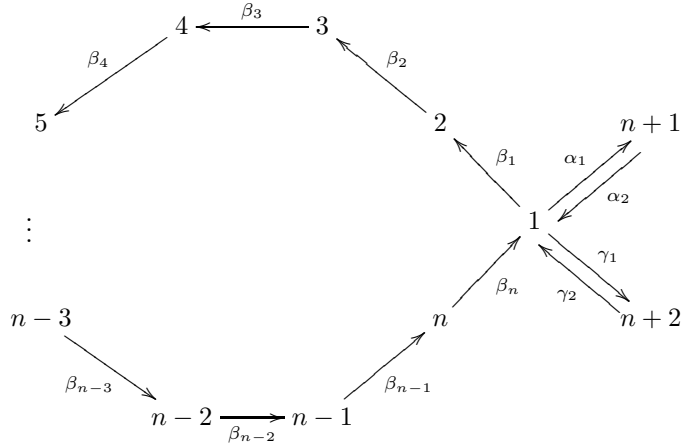
$$\xymatrix{
& & & & & & &\\
& & 4\ar[dl]_{\beta_4} & 3 \ar[l]_{\beta_3} &  &  & \\
&5 \ar@{}[dd]_{\vdots} & & & 2 \ar[ul]_{\beta_2}& & {n+1} \ar@<1.3 ex>[dl]^{\alpha_2}\\
&&&&& 1\ar[ul]_{\beta_1} \ar[ur]^{\alpha_1} \ar[dr]^{\gamma_1} & \\
& {n-3} \ar[dr]_{\beta_{n-3}} & & & n \ar[ur]_{\beta_n} & &  n+2 \ar@<1.3 ex>[ul]^{\gamma_2} \\
& & {n-2} \ar[r]_{\beta_{n-2}}  &{n-1}\ar[ur]_{\beta_{n-1}} & &&\\
}$$

\vspace{1cm}

Then $\Lambda = K{\mathcal Q}(n) / I(n)$ where $I(n)$ is the ideal generated by the relations:

(i) \hspace{3cm}$\alpha_1 \alpha_2 = (\beta_1 \beta_2 \cdots \beta_n)^2 = \gamma_1 \gamma_2,$

(ii)
$$ \beta_n\alpha_1= 0,  \hspace{1cm}
\beta_n \gamma_1 = 0,$$

$$\alpha_2 \beta_1= 0,  \hspace{1cm}
\gamma_2 \beta_1 = 0,$$

$$ \alpha_2 \alpha_1= 0,  \hspace{1cm}
\gamma_2 \gamma_1 = 0,$$

(iii)  for all $j\in \{2, \ldots, n\},$
$$\beta_j \beta_{j+1} \cdots  \beta_n \beta_1 \cdots \beta_n \beta_1 \cdots \beta_{j-1}\beta_j = 0.$$

\vspace{.5cm}
Note that we write our paths from left to right.

\bigskip

In order to compute $\HH^2(\L)$, the next
section gives the necessary background required to find the first terms of
the projective resolution of $\L$ as a $\L,\L$-bimodule.
Section \ref{sec3} and Section \ref{sec4} uses this part of a minimal
projective bimodule resolution for our algebras to determine the second Hochschild cohomology group and provides the main results
of this paper.

\bigskip
\section{projective resolutions} \label{sec2}
To find the second Hochschild cohomology group $\HH^2(\L)$, we could use the bar resolution given in \cite{H}. This bar resolution is not a minimal projective resolution of $\L$ as $\L, \L$-bimodule. In practice, it is easier to compute the Hochschild cohomology group if we use a minimal projective resolution. So here we use the
projective resolution of \cite{GS}. More generally, let $\L = K{\mathcal Q}/I$
be a finite dimensional algebra, where $K$ is an algebraically
closed field, ${\mathcal Q}$ is a quiver, and $I$ is an
admissible ideal of $K{\mathcal Q}$. Fix a minimal set $f^2$ of generators
for the ideal $I$. Let $x \in f^2$. Then $x =
\sum_{j=1}^{r} c_ja_{1j} \cdots a_{kj} \cdots a_{s_jj}$, that is, $x$ is a
linear combination of paths $a_{1j} \cdots a_{kj} \cdots a_{s_jj}$ for $j =
1, \ldots, r$ and $c_j \in K$ and there are unique vertices $v$ and $w$ such
that each path $a_{1j} \cdots a_{kj} \cdots a_{s_j j}$ starts at $v$ and
ends at $w$ for all $j$. We write $\mo(x) = v$ and $\mt(x) = w.$ Similarly
$\mo(a)$ is the origin of the arrow $a$ and $\mt(a)$ is the end of $a$.

In \cite[Theorem 2.9]{GS}, it is shown that there is a minimal
projective resolution of $\L$ as a $\L, \L$-bimodule which begins:
$$\cdots \rightarrow Q^3 \stackrel{A_3}{\rightarrow} Q^2 \stackrel{A_2}{\rightarrow} Q^1 \stackrel{A_1}{\rightarrow} Q^0 \stackrel{g}{\rightarrow} \L \rightarrow 0,$$
where the projective $\L, \L$-bimodules $Q^0, Q^1, Q^2$ are given by
$$Q^0 = \bigoplus_{v, vertex} \L v \otimes v\L,$$
$$\hspace*{1.5cm}Q^1 = \bigoplus_{a, arrow} \L \mo(a) \otimes \mt(a)\L, \mbox { and }$$
$$\hspace*{.5cm}Q^2 = \bigoplus_{x \in f^2} \L \mo(x) \otimes \mt(x) \L,$$
and the maps $g, A_1$, $A_2$ and $A_3$ are $\L, \L$-bimodule homomorphisms,
defined as follows. The map $g: Q^0 \rightarrow \L$ is the multiplication
map so is given by  $v \otimes v \mapsto v$. The map $A_1: Q^1\rightarrow Q^0$
is given by $\mo(a) \otimes \mt(a) \mapsto \mo(a) \otimes \mo(a) a - a \mt(a) \otimes \mt(a)$ for each arrow $a$.
With the notation for $x \in f^2$ given above, the map
$A_2: Q^2 \rightarrow Q^1$ is given by $\mo(x) \otimes \mt(x) \mapsto
\sum_{j=1}^{r}c_j(\sum_{k=1}^{s_j} a_{1j} \cdots a_{(k-1)j} \otimes
a_{(k+1)j} \cdots a_{s_j j})$, where $a_{1j} \cdots a_{(k-1)j} \otimes
a_{(k+1)j} \cdots a_{s_j j} \in \L \mo(a_{kj}) \otimes \mt(a_{kj})\L$.

In order to describe the projective bimodule $Q^3$ and the map $A_3$
in the $\L, \L$-bimodule resolution of $\L$ in \cite{GS}, we need to
introduce some notation from \cite{GSZ}. Recall that an
element $y \in K{\mathcal Q}$ is uniform if there are vertices $v, w$ such
that $y = v y = y w.$ We write $\mo(y) = v$ and $\mt(y) = w$. In \cite{GSZ}, Green,
Solberg and Zacharia show that there are sets $f^n$ in $K{\mathcal Q}$,
for $n \geq 3$, consisting of uniform
elements $y \in f^n$ such that $y = \sum_{x \in f^{n-1}} x r_x = \sum_{z \in
f^{n-2}} z s_z$ for unique elements $r_x, s_z \in K{\mathcal Q}$ such that
$s_z \in I$. These sets have
special properties related to a minimal projective $\L$-resolution of
$\L/\rrad$, where $\rrad$ is the Jacobson radical of $\L$. Specifically
the $n$-th projective in the minimal projective $\L$-resolution of $\L/\rrad$
is $\bigoplus_{y \in f^n} \mt(y) \L.$

In particular, to determine the set $f^3$, we follow explicitly the construction given in \cite[\S 1]{GSZ}. Let $f^1$ denote the set of arrows of ${\mathcal Q}$. Consider the intersection $(\bigoplus_i f^2_i K{\mathcal Q})\cap (\bigoplus_j f^1_jI)$. Set this intersection equal to some $(\bigoplus_l f^{3*}_l K{\mathcal Q})$. We then discard all elements of the form $f^{3*}$ that are in $\bigoplus_i f^2_iI$; the remaining ones form precisely the set $f^3$.

Thus, for $y \in f^3$ we have that $y \in (\bigoplus_i f^2_i K{\mathcal Q})\cap (\bigoplus_j f^1_jI)$.
So we may write $y = \sum f^2_i p_i = \sum q_i f^2_i r_i$ with $p_i, q_i, r_i \in
K{\mathcal Q}$, such that $p_i, q_i$ are in the ideal generated by the arrows of
$K{\mathcal Q}$, and $p_i$ unique. Then \cite{GS} gives that $Q^3 = \bigoplus_{y \in f^3} \L \mo(y) \otimes
\mt(y) \L$ and, for $y \in f^3$ in the notation
above, the component of $A_3 (\mo(y) \otimes \mt(y))$ in the summand $\L
\mo(f_i^2) \otimes \mt(f_i^2) \L$ of $Q^2$ is $\mo(y) \otimes p_i -
q_i \otimes r_i.$

Applying ${\Hom}(-, \L)$ to this part of a minimal projective bimodule
resolution of $\L$ gives us the complex
$$0 \rightarrow {\Hom}(Q^0, \L) \stackrel{d_1}{\rightarrow} {\Hom}(Q^1, \L) \stackrel{d_2}{\rightarrow} {\Hom}(Q^2, \L) \stackrel{d_3}{\rightarrow} {\Hom}(Q^3, \L)$$
where $d_i$ is the map induced from $A_i$ for $i = 1, 2, 3$. Then ${\HH}^2(\L) = {\Ker}\,d_3/{\Im}\,d_2.$

\bigskip

Throughout, all tensor products are tensor products over $K$, and we write
$\otimes$ for $\otimes_K$. When considering an element of the projective
$\L, \L$-bimodule $Q^1 = \bigoplus_{a, arrow} \L \mo(a) \otimes \mt(a) \L$
it is important to keep track of the individual summands of $Q^1$. So to
avoid confusion we usually denote an element in the summand $\L \mo(a)
\otimes \mt(a) \L$ by $\lambda \otimes_a \lambda'$ using the subscript `$a$'
to remind us in which summand this element lies. Similarly, an element
$\lambda \otimes_{f^2_i} \lambda'$ lies in the summand $\L \mo(f^2_i)
\otimes \mt(f^2_i) \L$ of $Q^2$  and an element $\lambda \otimes_{f^3_i}
\lambda'$ lies in the summand $\L \mo(f^3_i) \otimes \mt(f^3_i) \L$ of
$Q^3$. We keep this notation for the rest of the paper.

\bigskip

\section{${\HH}^2(\L)$ for $\L = \L(p,q,k,s,\lambda)$} \label{sec3}

We have given $\L = \L(p,q,k,s,\lambda)$ by quiver and relations in Section
\ref{sec1}. However, these relations are not minimal. So next we will find
a minimal set of relations $f^2$ for this algebra.

Let $$\begin{array}{lcl}
f^2_{1,1} &=& \alpha_{(1,0)} \alpha_{(1,1)} \cdots \alpha_{(1,q)} \beta_{1^0} \beta_{1^1} \cdots \beta_{1^p}\\
&&\hspace*{1cm} - \lambda \beta_{2^0} \beta_{2^1} \cdots \beta_{2^p} \alpha_{(s+2,0)} \alpha_{(s+2,1)} \cdots \alpha_{(s+2,q)},\\
f^2_{1,i} &=& \alpha_{(i,0)} \alpha_{(i,1)} \cdots \alpha_{(i,q)} \beta_{i^0} \beta_{i^1} \cdots \beta_{i^p} \\
&&\hspace*{1cm} - \beta_{(i+1)^0} \beta_{(i+1)^1} \cdots \beta_{(i+1)^p} \alpha_{(s+i+1,0)} \alpha_{(s+i+1,1)} \cdots \alpha_{(s+i+1,q)}\\
&&\hspace*{2.5cm} \mbox{for } i \in \{2, \ldots, k\},\\
f^2_{2,i} &=& \beta_{i^p} \beta_{(s+i+1)^0}\hspace*{.7cm} \mbox{for } i \in \{1, \ldots, k\},\\
f^2_{3,i} &=& \alpha_{(i,q)} \alpha_{(i-1,0)} \hspace*{.6cm} \mbox{for } i \in \{1, \ldots, k\},\\
f^2_{4,i,j} &=& \beta_{i^j} \beta_{i^{j+1}} \cdots \beta_{i^p} \alpha_{(s+i,0)} \alpha_{(s+i,1)} \cdots \alpha_{(s+i,q)} \beta_{(s+i)^0}\beta_{(s+i)^1} \cdots \beta_{(s+i)^j}\\
&&\hspace*{2.5cm} \mbox{where } j \in \{1, \ldots, p-1\} \mbox{ and } i \in \{1, \ldots, k\},\\
f^2_{5,i,t'} &=& \alpha_{(i,t')} \alpha_{(i,t'+1)} \cdots \alpha _{(i,q)} \beta_{i^0} \beta_{i^1} \cdots \beta_{i^p} \alpha_{(s+i,0)} \alpha_{(s+i,1)} \cdots \alpha_{(s+i,t')}\\
&&\hspace*{2.5cm} \mbox{where } t' \in \{1, \ldots, q-1\} \mbox{ and } i \in \{1, \ldots, k\}.\\
\end{array}$$

The remaining relations given in Section \ref{sec1} are all linear combinations of the above relations. For example, the relation
$\beta_{i^0} \beta_{i^1} \cdots \beta_{i^p} \alpha_{(s+i,0)} \alpha_{(s+i,1)} \cdots \alpha_{(s+i,q)} \beta_{(s+i)^0}$ can be written as
$\alpha_{(i-1,0)} \alpha_{(i-1,1)} \cdots \alpha_{(i-1,q)} \beta_{(i-1)^0}\beta_{(i-1)^1} \cdots \beta_{(i-1)^{p-1}} f^2_{2,i-1} - f^2_{1,i-1}$ $\beta_{(s+i)^0}.$
So this relation is in $I$ and is not in $f^2$.

\begin{prop}
For $\L = \L(p, q, k, s, \lambda)$ and with the above notation, the minimal set of relations is
$$f^2 = \{f^2_{1,i}, f^2_{2,i}, f^2_{3,i}, f^2_{4,i,j}, f^2_{5,i,t'}\}.$$
\end{prop}

\bigskip

In contrast to the majority of self-injective algebras of finite
representation type, we will show that the algebra
$\L(p, q, k, s,\lambda)$ has non-zero second Hochschild cohomology group
(see \cite[Theorem 6.5]{D}). Recall that ${\HH}^2(\L) = {\Ker}\,d_3 /
{\Im}\,d_2$, where $d_3: {\Hom}(Q^2, \L) \rightarrow {\Hom}(Q^3, \L)$ is induced by
$A_3 : Q^3 \to Q^2$.

First we will find ${\Im}\,d_2$. Since $d_2: {\Hom}(Q^1, \L) \rightarrow {\Hom}(Q^2, \L)$, let $f \in {\Hom}(Q^1, \L)$ so that $d_2f = fA_2$. We consider the cases $1 \leq s \leq k-2$ and $s = k-1$ separately.
\vspace*{.5cm}

Let $1 \leq s \leq k-2$ and

\hspace*{1.5cm}$f(e_i \otimes_{\beta_{(i+1)^0}} e_{(i+1)^1}) = c_{1,i} \beta_{(i+1)^0},$

\hspace*{1.5cm}$f(e_{(i+1)^j} \otimes_{\beta_{(i+1)^j}} e_{(i+1)^{j+1}}) = c_{2,i+1,j} \beta_{(i+1)^j} \mbox{ for }j \in \{1, \ldots, p-1\},$

\hspace*{1.5cm}$f(e_{(i+1)^p} \otimes_{\beta_{(i+1)^p}} e_{s+i+1}) = c_{2,i+1,p} \beta_{(i+1)^p},$

\hspace*{1.5cm}$f(e_i \otimes_{\alpha_{(i,0)}} e_{(i,1)}) = c_{3,i} \alpha_{(i,0)},$

\hspace*{1.5cm}$f(e_{(i,t')} \otimes_{\alpha_{(i,t')}} e_{(i,t'+1)}) = c_{4,i,t'} \alpha_{(i,t')} \mbox{ for } t' \in \{1, \ldots, q-1\} \mbox {  and }$

\hspace*{1.5cm}$f(e_{(i,q)} \otimes_{\alpha_{(i,q)}} e_{i-1}) = c_{4,i,q} \alpha_{(i,q)},$\\
where all coefficients $c_{1,i}, c_{2,i+1,j} \mbox{ for } j \in \{1, \ldots, p-1\}, c_{2,i+1,p}, c_{3,i}, c_{4,i,t'} \mbox{ for } t' \in \{1, \ldots, q-1\}, c_{4,i,q}$ $\in K.$ Now we find $fA_2$.

First we have, $fA_2(e_1 \otimes_{f^2_{1,1}} e_{s+1})$
$$\begin{array}{lcl}
&=& f(e_1 \otimes_{\alpha_{(1,0)}} e_{(1,1)}) \alpha_{(1,1)} \cdots \alpha_{(1,q)} \beta_{1^0} \beta_{1^1} \cdots \beta_{1^p} \\
& &+ \alpha_{(1,0)} f(e_{(1,1)} \otimes_{\alpha_{(1,1)}} e_{(1,2)}) \alpha_{(1,2)} \cdots \alpha_{(1,q)} \beta_{1^0} \beta_{1^1} \cdots \beta_{1^p} \\
&&+ \cdots +  \alpha_{(1,0)} \alpha_{(1,1)} \cdots \alpha_{(1,q-1)} f(e_{(1,q)} \otimes_{\alpha_{(1,q)}} e_{k}) \beta_{1^0} \beta_{1^1} \cdots \beta_{1^p}\\
& &+ \alpha_{(1,0)} \alpha_{(1,1)} \cdots \alpha_{(1,q)} f(e_{k} \otimes_{\beta_{1^0}} e_{1^1}) \beta_{1^1} \cdots \beta_{1^p}\\
& &+ \alpha_{(1,0)} \alpha_{(1,1)} \cdots \alpha_{(1,q)} \beta_{1^0} f(e_{1^1} \otimes_{\beta_{1^1}} e_{1^2}) \beta_{1^2} \cdots \beta_{1^p}\\
& &+ \cdots + \alpha_{(1,0)} \alpha_{(1,1)} \cdots \alpha_{(1,q)} \beta_{1^0} \beta_{1^1} \cdots \beta_{1^{p-1}} f(e_{1^p} \otimes_{\beta_{1^p}} e_{s+1})\\
& &- \lambda [f(e_1 \otimes_{\beta_{2^0}} e_{2^1}) \beta_{2^1} \cdots \beta_{2^p} \alpha_{(s+2,0)} \alpha_{(s+2,1)} \cdots \alpha_{(s+2,q)} \\
& &+ \beta_{2^0} f(e_{2^1} \otimes_{\beta_{2^1}} e_{2^2}) \beta_{2^2} \cdots \beta_{2^p} \alpha_{(s+2,0)} \alpha_{(s+2,1)} \cdots \alpha_{(s+2,q)}\\
& &+ \cdots + \beta_{2^0} \beta_{2^1} \cdots \beta_{2^{p-1}} f(e_{2^p} \otimes_{\beta_{2^p}} e_{s+2}) \alpha_{(s+2,0)} \alpha_{(s+2,1)} \cdots \alpha_{(s+2,q)}\\
& &+  \beta_{2^0} \beta_{2^1} \cdots \beta_{2^p} f(e_{s+2} \otimes_{\alpha_{(s+2,0)}} e_{(s+2,1)}) \alpha_{(s+2,1)} \cdots \alpha_{(s+2,q)}\\
& &+  \beta_{2^0} \beta_{2^1} \cdots \beta_{2^p} \alpha_{(s+2,0)} f(e_{(s+2,1)} \otimes_{\alpha_{(s+2,1)}} e_{(s+2,2)}) \alpha_{(s+2,2)} \cdots \alpha_{(s+2,q)}\\
& &+ \cdots +  \beta_{2^0} \beta_{2^1} \cdots \beta_{2^p} \alpha_{(s+2,0)} \alpha_{(s+2,1)} \cdots \alpha_{(s+2,q-1)} f(e_{(s+2,q)} \otimes_{\alpha_{(s+2,q)}} e_{s+1})]\\
&=& (c_{3,1} + c_{4,1,1} + \cdots + c_{4,1,q} + c_{1,k} + c_{2,1,1} + \cdots + c_{2,1,p})
\alpha_{(1,0)} \alpha_{(1,1)} \cdots \alpha_{(1,q)} \beta_{1^0}\\
&& \beta_{1^1} \cdots \beta_{1^p}\\
&& -\lambda (c_{1,1} + c_{2,2,1} + \cdots + c_{2,2,p} + c_{3,s+2} + c_{4,s+2,1} + \cdots + c_{4,s+2,q})
\beta_{2^0} \beta_{2^1} \cdots \beta_{2^p}\\
&&\alpha_{(s+2,0)}\alpha_{(s+2,1)} \cdots \alpha_{(s+2,q)}\\
&=& (c_{3,1} + c_{4,1,1} + \cdots + c_{4,1,q} + c_{1,k} + c_{2,1,1} + \cdots + c_{2,1,p} - c_{1,1} - c_{2,2,1} - \cdots -\\
&& c_{2,2,p} - c_{3,s+2}
 - c_{4,s+2,1} - \cdots - c_{4,s+2,q})\alpha_{(1,0)} \alpha_{(1,1)} \cdots \alpha_{(1,q)} \beta_{1^0} \beta_{1^1} \cdots \beta_{1^p}.
\end{array}$$

Similarly for $i \in \{2, \ldots, k\}$, $fA_2(e_i \otimes_{f^2_{1,i}} e_{s+i})$
$= (c_{3,i} + c_{4,i,1} + \cdots + c_{4,i,q} + c_{1,i-1} + c_{2,i,1} + \cdots + c_{2,i,p} - c_{1,i} - c_{2,i+1,1} - \cdots - c_{2,i+1,p} - c_{3,s+i+1} - c_{4,s+i+1,1} - \cdots - c_{4,s+i+1,q})
\alpha_{(i,0)} \alpha_{(i,1)} \cdots \alpha_{(i,q)} \beta_{i^0} \beta_{i^1} \cdots \beta_{i^p}.$

For the remaining terms, $fA_2(\mo(x) \otimes_{x} \mt(x)) = 0$ where  $x \in \{f^2_{2,i}, f^2_{3,i},f^2_{4,i,j},f^2_{5,i,t'}\}$
for all $i \in \{1, \ldots, k\}$, $j \in \{1, \ldots, p-1\}$ and $t' \in \{1, \ldots, q-1\}$.

\vspace{.5cm}

Let $c_i' = c_{3,i} + c_{4,i,1} + \cdots + c_{4,i,q} + c_{1,i-1} + c_{2,i,1} + \cdots + c_{2,i,p} - c_{1,i} - c_{2,i+1,1} -
\cdots - c_{2,i+1,p} - c_{3,s+i+1} - c_{4,s+i+1,1} - \cdots - c_{4,s+i+1,q}$ for $i = 1, \ldots, k$\\

and $\rho_i = \alpha_{(i,0)} \alpha_{(i,1)} \cdots \alpha_{(i,q)} \beta_{i^0} \beta_{i^1} \cdots \beta_{i^p}$ for $i = 1, \ldots, k.$\\

Thus for $i \in \{1, \ldots, k\}$ and $1 \leq s \leq k-2$, $fA_2$ is given by

\hspace*{2cm}$fA_2(e_i \otimes_{f^2_{1,i}} e_{s+i}) = c_i' \rho_i,$

\hspace*{2cm}$fA_2(e_{i^p} \otimes_{f^2_{2,i}} e_{(s+i+1)^1}) = 0,$

\hspace*{2cm}$fA_2(e_{(i,q)} \otimes_{f^2_{3,i}} e_{(i-1,1)})  = 0,$

\hspace*{2cm}$fA_2(e_{i^j} \otimes_{f^2_{4,i,j}} e_{(s+i)^{j+1}}) = 0 \mbox{ where } j \in  \{1, \ldots, p-1\} \mbox { and }$

\hspace*{2cm}$fA_2(e_{(i,t')} \otimes_{f^2_{5,i,t'}} e_{(s+i,t'+1)}) = 0  \mbox{ where } t' \in \{1, \ldots, q-1\},$\\
where $c_1', \ldots, c_k' \in K$ with $\Sigma_{i=1}^k c_i' = 0$. So ${\dim}\,{\Im}\,d_2 = k - 1.$

\vspace*{.5cm}

For $s =k-1$, we let

\hspace*{1.5cm}$f(e_i \otimes_{\beta_{(i+1)^0}} e_{(i+1)^1})  =  c_{1,i} \beta_{(i+1)^0},$

\hspace*{1.5cm}$f(e_{(i+1)^j} \otimes_{\beta_{(i+1)^j}} e_{(i+1)^{j+1}})  =  c_{2,i+1,j} \beta_{(i+1)^j} \mbox{ for }j \in \{1, \ldots, p-1\},$

\hspace*{1.5cm}$f(e_{(i+1)^p} \otimes_{\beta_{(i+1)^p}} e_i)  =  c_{2,i+1,p} \beta_{(i+1)^p},$

\hspace*{1.5cm}$f(e_i \otimes_{\alpha_{(i,0)}} e_{(i,1)})  =  c_{3,i} \alpha_{(i,0)},$

\hspace*{1.5cm}$f(e_{(i,t')} \otimes_{\alpha_{(i,t')}} e_{(i,t'+1)})  = c_{4,i,t'} \alpha_{(i,t')} \mbox{ for } t' \in \{1, \ldots, q-1\} \mbox {  and }$

\hspace*{1.5cm}$f(e_{(i,q)} \otimes_{\alpha_{(i,q)}} e_{i-1})  =  c_{4,i,q} \alpha_{(i,q)} + d_{1,i} \alpha_{(i,q)}\beta_{i^0} \beta_{i^1} \cdots \beta_{i^p},$\\
where for all $i \in \{1,\ldots, k\}$ the coefficients $c_{1,i}, c_{2,i+1,j} \mbox{ for } j \in \{1, \ldots, p-1\}, c_{2,i+1,p},$ $c_{3,i}, c_{4,i,t'} \mbox{ for } t' \in \{1, \ldots, q-1\}, c_{4,i,q},$ $d_{1,i}$ are in $K.$

\vspace*{.3cm}

Then we can find $fA_2$ for $i \in \{1, \ldots, k\}$ in the same way as the previous case to see that it is given by

\hspace*{1cm}$fA_2(e_i \otimes_{f^2_{1,i}} e_{i-1}) = c_i' \rho_i \mbox{ where } c_i', \rho_i \mbox{ as above },$

\hspace*{1cm}$fA_2(e_{i^p} \otimes_{f^2_{2,i}} e_{i^1}) = 0,$

\hspace*{1cm}$fA_2(e_{(i,q)} \otimes_{f^2_{3,i}} e_{(i-1,1)})  = d_{1,i} \alpha_{(i,q)}\beta_{i^0} \beta_{i^1} \cdots \beta_{i^p} \alpha_{(i-1,0)},$

\hspace*{1cm}$fA_2(e_{i^j} \otimes_{f^2_{4,i,j}} e_{(i-1)^{j+1}}) = 0 \mbox{ where } j \in  \{1, \ldots, p-1\} \mbox { and }$

\hspace*{1cm}$fA_2(e_{(i,t')} \otimes_{f^2_{5,i,t'}} e_{(i-1,t'+1)}) = 0  \mbox{ where } t' \in \{1, \ldots, q-1\},$\\
where $c_1', \ldots, c_k',d_{1,1}, \ldots, d_{1,k} \in K$ with $\Sigma_{i=1}^k c_i' = 0$. Note that there is no dependency between the $d_{1,i}.$ So ${\dim}\,{\Im}\,d_2 = 2k - 1.$

\begin{prop} \label{pro2}
If $1 \leq s \leq k-2$, we have ${\dim}\,{\Im}\,d_2 = k - 1.$
If $s = k-1$, we have ${\dim}\,{\Im}\,d_2 = 2k -1.$
\end{prop}

\vspace*{.3cm}
Next we find ${\Hom(Q^2,\L)}$ and again consider the two cases separately.
Let $1 \leq s \leq k-2$ and $h \in {\Hom}(Q^2, \L)$. Then $h$ is defined by
$$\begin{array}{rcl}
\mo(f^2_{1,i}) \otimes \mt(f^2_{1,i}) & \mapsto & d_i \alpha_{(i,0)} \alpha_{(i,1)} \cdots \alpha_{(i,q)} \beta_{i^0} \beta_{i^1} \cdots \beta_{i^p}  \mbox{   for } i \in \{1,2, \ldots, k\},\\
\mbox{else } & \mapsto & 0,\\
\end{array}$$ where $d_i \in K$.

Therefore ${\dim}\,{\Hom(Q^2,\L)} = k.$ Hence, ${\dim}\,{\Ker}\,d_3 \leq k.$

\vspace*{.5cm}
For $s= k-1$ and $i \in \{1,2, \ldots, k\}$, $h$ is given by
$$\begin{array}{rcl}
\mo(f^2_{1,i}) \otimes \mt(f^2_{1,i}) & \mapsto & d_i \alpha_{(i,0)} \alpha_{(i,1)} \cdots \alpha_{(i,q)} \beta_{i^0} \beta_{i^1} \cdots \beta_{i^p} + d_i' \alpha_{(i,0)} \alpha_{(i,1)} \cdots \alpha_{(i,q)},\\
\mo(f^2_{3,i}) \otimes \mt(f^2_{3,i}) & \mapsto & d_i''  \alpha_{(i,q)} \beta_{i^0} \beta_{i^1} \cdots \beta_{i^p} \alpha_{(i-1,0)},\\
\mbox{else } & \mapsto & 0,\\
\end{array}$$ where $d_i, d_i', d_i''$ are in $K$ for $i \in \{1,\ldots,k\}.$ Thus ${\dim}\,{\Hom(Q^2,\L)} = 3k.$

\begin{prop}\label{bb1}
If $1 \leq s \leq k-2$, we have ${\dim}\,{\Hom(Q^2,\L)} = k.$
If $s = k-1$, ${\dim}\,{\Hom(Q^2,\L)} = 3k.$
\end{prop}

\begin{cor} \label{cor1}
If $1 \leq s \leq k-2$, we have ${\dim}\,{\Ker}\,d_3 \leq k.$
If $s = k-1$, ${\dim}\,{\Ker}\,d_3 \leq 3k.$
\end{cor}

\vspace*{.3cm}

In order to find ${\Ker}\,d_3$ and hence determine ${\HH}^2(\L)$ we start by giving a non-zero element in ${\HH}^2(\L)$ for all $s$.

\begin{prop} \label{b1}
Define $h_1 \in {\Hom}(Q^2, \L)$ by
$$\begin{array}{rcl}
\mo(f^2_{1,1}) \otimes \mt(f^2_{1,1}) = e_1 \otimes e_{s+1} & \mapsto & \alpha_{(1,0)} \alpha_{(1,1)} \cdots \alpha_{(1,q)} \beta_{1^0} \beta_{1^1} \cdots \beta_{1^p} = \rho_1,\\
\mbox{else } & \mapsto & 0.\\
\end{array}$$
Then $h_1$ is in $\Ker\, d_3$.
\end{prop}

\begin{proof}
We note that $\rho_1 \neq 0$ so $h_1$ is a non-zero map.
To show that $h_1 \in {\Ker}\,d_3$ we show that $h_1A_3 = 0$.
First, observe that $\rho_1 \beta_{(s+2)^0} = 0$
and $\rho_1 \alpha_{(s+1,0)}= 0.$ Hence $\rho_1 {\rrad} = 0$. Similarly we have ${\rrad} \rho_1 = 0.$

Recall that $Q^3 = \coprod_{y \in f^3} \Lambda \mo(y) \otimes \mt(y) \Lambda$ where $y = \sum_u f^2_u p_u = \sum_u q_u f^2_u r_u$ and $p_u, q_u$ are in the ideal generated by the arrows. For $y \in f^3$ the component of $A_3(\mo(y) \otimes \mt(y))$ in $\Lambda \mo(f_{u}^{2}) \otimes \mt(f_{u}^{2}) \Lambda$ is
$$\Sigma (\mo(y) \otimes_{f^2_u} p_u - q_u \otimes_{f^2_u} r_u).$$ Then $$h_1A_3 (\mo(y) \otimes \mt(y)) = \Sigma_u (h_1(\mo(y) \otimes_{f^2_u} p_u) - q_u h_1(\mo(f^2_u) \otimes_{f^2_u} \mt(f^2_u)) r_u).$$
Thus $h_1(\mo(y) \otimes_{f^2_u} p_u) = \left\{
\begin{array}{ll}
\rho_1 p_u & \mbox{ if } f^2_u = f^2_{1,1}\\
0 & \mbox{ otherwise}.\\
\end{array}\right.$

As $p_u$ is in the arrow ideal of $K{\mathcal Q}$, $\rho_1 p_u \in \rho_1 {\rrad} = 0.$ So we have $h_1(\mo(y) \otimes p_u) = 0.$ Similarly $h_1(q_u \otimes_{f^2_u} r_u) = 0$ as $q_u \rho_1 r_u \in {\rrad} \rho_1 r_u = 0.$
Therefore $h_1A_3(\mo(y) \otimes \mt(y)) = 0$ for all $y \in f^3$ so $h_1A_3 = 0$. Thus $h_1 \in
{\Ker}\,d_3$ as required.
\end{proof}

\begin{thm} \label{thm1}
For $\L = \L(p,q, k, s,\lambda)$ where $p, q$ are positive integers, $k \geq 2$, $1 \leq s \leq k-1$ with ${\gcd}(s+2, k) = 1 = {\gcd}(s, k)$ and  $\lambda \in K\setminus\{0\}$, we have ${\HH}^2(\L) \neq 0$.
\end{thm}

\begin{proof}
Consider the element $h_1 + {\Im}\,d_2$ of ${\HH}^2(\L)$ where $h_1$ is given as in Proposition \ref{b1} by
$$\begin{array}{rcl}
\mo(f^2_{1,1}) \otimes \mt(f^2_{1,1}) = e_1 \otimes e_{s+1} & \mapsto & \rho_1,\\
\mbox{else } & \mapsto & 0.\\
\end{array}$$

Suppose for contradiction that $h_1 \in {\Im}\,d_2.$ Then $h_1(e_1 \otimes e_{s+1}) = fA_2(e_1 \otimes e_{s+1})$. So $\rho_1 = c'_1 \rho_1$ and so $c'_1 = 1$. Also $h_1(e_i \otimes e_{s+i}) = fA_2(e_i \otimes e_{s+i})$ where $i \in \{2, \ldots, k\}.$ Then $0 = c'_i \rho_i,$ where $i \in \{2, \ldots, k\}.$ But this contradicts having  $\Sigma_{i=1}^k c_i' = 0$. Therefore $h_1 \notin {\Im}\,d_2$, that is, $h_1 + {\Im}\,d_2 \neq 0 + {\Im}\,d_2$. So $h_1 + {\Im}\,d_2$ is a non-zero element in ${\HH}^2(\L).$
\end{proof}

Note that we can also define maps $h_i: Q^2 \rightarrow \L$ by
$$\begin{array}{rcl}
\mo(f^2_{1,i}) \otimes \mt(f^2_{1,i}) & \mapsto & \alpha_{(i,0)} \alpha_{(i,1)} \cdots \alpha_{(i,q)} \beta_{i^0} \beta_{i^1} \cdots \beta_{i^p}= \rho_i,\\
\mbox{else } & \mapsto & 0.\\
\end{array}$$
for $i = 2, \ldots, k$. However, $h_1, h_2, \ldots, h_k$ all represent the same element $h_1 + {\Im}\,d_2$ of ${\HH}^2(\L)$.

\vspace*{.3cm}
As we have found a non-zero element in ${\HH}^2(\L)$ we know that ${\dim}\,{\HH^2(\L)} \geq 1$. In the case $1 \leq s \leq k-2$ we have the following result, the proof of which is immediate from Proposition \ref{pro2}, Corollary \ref{cor1} and Theorem \ref{thm1}.

\begin{prop} \label{pro4}
For $\L = \L(p,q, k, s,\lambda)$ where $1 \leq s \leq k-2$, we have ${\dim}\,{\Ker}\,d_3 = k$ and ${\dim}\,{\HH^2(\L)} = 1.$
\end{prop}

\vspace*{.5cm}

For the case $s = k-1$, we need more details to find ${\Ker}\,d_3$. Following \cite{GSZ} we may choose the set $f^3$ to consist of the following elements:
$$\{f^3_{1,i}, f^3_{2,i}, f^3_{3,i,t'},f^3_{4,i,j}\},\mbox{ where }$$

$$\begin{array}{l c l l l l l l}
f^3_{1,i}= f^2_{1,i}  \alpha_{(i-1,0)} \alpha_{(i-1,1)} + \beta_{(i+1)^0} \beta_{(i+1)^1} \cdots \beta_{(i+1)^p} \alpha_{(i,0)} \alpha_{(i,1)} \cdots \alpha_{(i,q-1)} f^2_{3,i} \alpha_{(i-1,1)} \\
\end{array}$$
$$\begin{array}{l c l l l l l l}
 &=& \alpha _{(i,0)} f^2_{5,i,1} \\
&& \hspace*{.8cm} \in e_i K{\mathcal Q}e_{(i-1,2)} \mbox{ where $i \in \{2, \ldots, k\}$},\\
f^3_{1,1}&=& f^2_{1,1}  \alpha_{(k,0)} \alpha_{(k,1)} + \lambda \beta_{2^0} \beta_{2^1} \cdots \beta_{2^p} \alpha_{(1,0)} \alpha_{(1,1)} \cdots \alpha_{(1,q-1)} f^2_{3,1} \alpha_{(k,1)} \\
      &=& \alpha _{(1,0)} f^2_{5,1,1}\\
      && \hspace*{.8cm} \in e_1 K{\mathcal Q}e_{(k,2)}, \\
f^3_{2,i} &=& f^2_{1,i} \beta_{i^0} \beta_{i^1} -  \alpha_{(i,0)} \alpha_{(i,1)} \cdots \alpha_{(i,q)} \beta_{i^0} \beta_{i^1} \cdots \beta_{i^{(p-1)}}f^2_{2,i} \beta_{i^1}\\
&=& - \beta_{(i+1)^0} f^2_{4,i+1,1} \\
&& \hspace*{.8cm}\in e_i K{\mathcal Q}e_{i^2} \mbox{ where $i \in \{2, \ldots, k\}$},\\
f^3_{2,1} &=& f^2_{1,1} \beta_{1^0} \beta_{1^1} -  \alpha_{(1,0)} \alpha_{(1,1)} \cdots \alpha_{(1,q)} \beta_{1^0} \beta_{1^1} \cdots \beta_{1^{(p-1)}}f^2_{2,1} \beta_{1^1}\\
&=& - \lambda \beta_{2^0} f^2_{4,2,1} \\
&&\hspace*{.8cm} \in e_1 K{\mathcal Q}e_{1^2},\\
f^3_{3,i,t'} &=& f^2_{5,i,t'} \alpha_{(i-1,t'+1)}\\
&=& \alpha_{(i,t')} f^2_{5,i,t'+1}\\
 &&\hspace*{.8cm} \in e_{(i,t')} K{\mathcal Q}e_{(i-1,t'+2)} \mbox{ where $i \in \{1, \ldots, k\}$ and $t' \in \{1, \ldots, q-2\}$},\\
f^3_{3,i,q-1} &=& f^2_{5,i,q-1} \alpha_{(i-1,q)} - \alpha_{(i,q-1)} f^2_{3,i} \alpha_{(i-1,1)} \cdots \alpha_{(i-1,q)}\beta_{(i-1)^0} \beta_{(i-1)^1} \cdots \beta_{(i-1)^p}\\
&=& - \alpha_{(i,q-1)} \alpha_{(i,q)} f^2_{1,i-1}\\
&& \hspace*{.8cm} \in e_{(i,q-1)} K{\mathcal Q}e_{i-2} \mbox{ where $i \in \{1,3, \ldots, k\}$},\\
f^3_{3,2,q-1} &=& \lambda f^2_{5,2,q-1} \alpha_{(1,q)} - \alpha_{(2,q-1)} f^2_{3,2} \alpha_{(1,1)} \cdots \alpha_{(1,q)}\beta_{1^0} \beta_{1^1} \cdots \beta_{1^p}\\
&=& - \alpha_{(2,q-1)} \alpha_{(2,q)} f^2_{1,1} \\
&&\hspace*{.8cm} \in e_{(2,q-1)} K{\mathcal Q}e_k,\\
\end{array}$$
$$\begin{array}{l c l l l l l l}
f^3_{3,i,q} = f^2_{3,i}  \alpha_{(i-1,1)} \cdots \alpha_{(i-1,q)} \beta_{(i-1)^0} \beta_{(i-1)^1} \cdots \beta_{(i-1)^p} \alpha_{(i-2,0)} - \alpha_{(i,q)} f^2_{1,i-1} \alpha_{(i-2,0)}\\
\end{array}$$
$$\begin{array}{l c l l l l l l}
&=& \alpha_{(i,q)} \beta_{i^0} \beta_{i^1} \cdots \beta_{i^p} \alpha_{(i-1,0)} \alpha_{(i-1,1)}\cdots \alpha_{(i-1,q-1)} f^2_{3,i-1}\\
 && \hspace*{.8cm} \in e_{(i,q)} K{\mathcal Q}e_{(i-2,1)} \mbox{ where $i \in \{1,3, \ldots, k\}$},\\
f^3_{3,2,q} &=& f^2_{3,2}  \alpha_{(1,1)} \cdots \alpha_{(1,q)} \beta_{1^0} \beta_{1^1} \cdots \beta_{1^p} \alpha_{(k,0)} - \alpha_{(2,q)} f^2_{1,1} \alpha_{(k,0)}\\
&=& \lambda \alpha_{(2,q)} \beta_{2^0} \beta_{2^1} \cdots \beta_{2^p} \alpha_{(1,0)} \alpha_{(1,1)}\cdots \alpha_{(1,q-1)} f^2_{3,1} \\
&&\hspace*{.8cm} \in e_{(2,q)} K{\mathcal Q}e_{(k,1)},\\
f^3_{4,i,j} &=& f^2_{4,i,j} \beta_{(i-1)^{(j+1)}} \\
&=& \beta_{i^j} f^2_{4,i,j+1}\\
 &&\hspace*{.8cm} \in e_{i^j} K{\mathcal Q}e_{(i-1)^{(j+2)}} \mbox{ where $i \in \{1, \ldots, k\}$ and $j \in \{1, \ldots, p-2\}$},\\
f^3_{4,i,p-1} &=& f^2_{4,i,p-1} \beta_{(i-1)^p} - \beta_{i^{(p-1)}} f^2_{2,i} \beta_{i^1} \cdots \beta_{i^p} \alpha_{(i-1,0)} \alpha_{(i-1,1)}\cdots \alpha_{(i-1,q)}\\
&=&  \beta_{i^{(p-1)}} \beta_{i^p} f^2_{1,i-1} \\
&&\hspace*{.8cm} \in  e_{i^{(p-1)}} K{\mathcal Q}e_{i-2} \mbox{ where $i \in \{1,3, \ldots, k\}$},\\
f^3_{4,2,p-1} &=& f^2_{4,2,p-1} \beta_{1^p} - \lambda \beta_{2^{(p-1)}} f^2_{2,2} \beta_{2^1} \cdots \beta_{2^p} \alpha_{(1,0)} \alpha_{(1,1)}\cdots \alpha_{(1,q)}\\
&=&  \beta_{2^{(p-1)}} \beta_{2^p} f^2_{1,1} \\
&&\hspace*{.8cm} \in  e_{2^{(p-1)}} K{\mathcal Q}e_k,\\
f^3_{4,i,p} &=& f^2_{2,i} \beta_{i^1} \cdots \beta_{i^p} \alpha_{(i-1,0)} \alpha_{(i-1,1)}\cdots \alpha_{(i-1,q)} \beta_{(i-1)^0} + \beta_{i^p} f^2_{1,i-1} \beta_{(i-1)^0}\\
&=&  \beta_{i^p} \alpha_{(i-1,0)} \alpha_{(i-1,1)}\cdots \alpha_{(i-1,q)}\beta_{(i-1)^0} \beta_{(i-1)^1} \cdots \beta_{(i-1)^{(p-1)}} f^2_{2,i-1}\\
&& \hspace*{.8cm}\in  e_{i^p} K{\mathcal Q}e_{(i-1)^1} \mbox{ where $i \in \{1,3, \ldots, k\}$},\\
 \end{array}$$
\newpage
$$\begin{array}{l c l l l l l l}
 f^3_{4,2,p} &=& \lambda f^2_{2,2} \beta_{2^1} \cdots \beta_{2^p} \alpha_{(1,0)} \alpha_{(1,1)}\cdots \alpha_{(1,q)} \beta_{1^0} + \beta_{2^p} f^2_{1,1} \beta_{1^0}\\
&=&  \beta_{2^p} \alpha_{(1,0)} \alpha_{(1,1)}\cdots \alpha_{(1,q)}\beta_{1^0} \beta_{1^1} \cdots \beta_{1^{(p-1)}} f^2_{2,1} \\
&&\hspace*{.8cm} \in  e_{2^p} K{\mathcal Q}e_{1^1}.
\end{array}$$

Thus the projective bimodule $Q^3$ is $ \bigoplus_{y \in f^3} \L \mo(y) \otimes \mt(y) \L$\\
$= \bigoplus_{i=1}^{k}[(\L e_i \otimes_{f^3_{1,i}} e_{(i-1,2)}\L) \oplus (\L  e_i \otimes_{f^3_{2,i}} e_{i^2}\L)$
$\bigoplus_{t'=1}^{q-2}(\L e_{(i,t')} \otimes_{f^3_{3,i,t'}} e_{(i-1,t'+2)} \L) \\
\hspace*{.3cm}\oplus (\L  e_{(i,q-1)}\otimes_{f^3_{3,i,q-1}} e_{i-2} \L) \oplus (\L e_{(i,q)} \otimes_{f^3_{3,i,q}} e_{(i-2,1)}\L)$
$\bigoplus_{j=1}^{p-2} (\L e_{i^j} \otimes_{f^3_{4,i,j}} e_{(i-1)^{(j+2)}}\L)\\
\hspace*{.2cm}\oplus (\L e_{i^{(p-1)}} \otimes_{f^3_{4,i,p-1}} e_{i-2}\L) \oplus (\L e_{i^p} \otimes_{f^3_{4,i,p}} e_{(i-1)^1}\L) ].$

\vspace*{.5cm}
Now we determine $\Ker\,d_3$ in the case $s=k-1$. Let $h \in {\Ker}\,d_3$, so $h \in
{\Hom}(Q^2, \L)$ and $d_3h = 0$. Recall that for $i \in \{1,\ldots,k\}$, $h$ is given by
$$\begin{array}{rcl}
\mo(f^2_{1,i}) \otimes \mt(f^2_{1,i}) & \mapsto & d_i \alpha_{(i,0)} \alpha_{(i,1)} \cdots \alpha_{(i,q)} \beta_{i^0} \beta_{i^1} \cdots \beta_{i^p} + d_i' \alpha_{(i,0)} \alpha_{(i,1)} \cdots \alpha_{(i,q)},\\
\mo(f^2_{3,i}) \otimes \mt(f^2_{3,i}) & \mapsto & d_i''  \alpha_{(i,q)} \beta_{i^0} \beta_{i^1} \cdots \beta_{i^p} \alpha_{(i-1,0)},\\
\mbox{else } & \mapsto & 0,\\
\end{array}$$ where $d_i, d_i', d_i''$ are in $K$.

Then for $i \in \{1, \ldots, k\}$, we have $hA_3(e_i \otimes_{f^3_{1,i}} e_{(i-1,2)})$
$$\begin{array}{lcl}
& = & h(e_i \otimes_{f^2_{1,i}} e_{i-1}) \alpha_{(i-1,0)} \alpha_{(i-1,1)} \\
&&+\beta_{(i+1)^0} \beta_{(i+1)^1} \cdots \beta_{(i+1)^p} \alpha_{(i,0)}\alpha_{(i,1)} \cdots \alpha_{(i,q-1)} h(e_{(i,q)} \otimes_{f^2_{3,i}} \alpha_{(i-1,1)}) \\
&&- \alpha_{(i,0)} h(e_{(i,1)} \otimes_{f^2_{5,i,1}} e_{(i-1,2)})\\
 &=& d_i \alpha_{(i,0)} \alpha_{(i,1)} \cdots \alpha_{(i,q)}\beta_{i^0} \beta_{i^1} \cdots \beta_{i^p} \alpha_{(i-1,0)} \alpha_{(i-1,1)} \\
 &&+ d_i'\alpha_{(i,0)} \alpha_{(i,1)} \cdots \alpha_{(i,q)} \alpha_{(i-1,0)} \alpha_{(i-1,1)}\\
&& + d_i''  \beta_{(i+1)^0} \beta_{(i+1)^1} \cdots \beta_{(i+1)^p}\alpha_{(i,0)} \alpha_{(i,1)}\cdots \alpha_{(i,q)}\beta_{i^0} \beta_{i^1} \cdots \beta_{i^p} \alpha_{(i-1,0)}\\
&=& 0.
\end{array}$$
In a similar way we can show that $hA_3(e_1 \otimes_{f^3_{1,1}} e_{(k,2)}) =0$.

For $i \in \{2, \ldots, k\}$, we have $hA_3(e_i \otimes_{f^3_{2,i}} e_{i^2})$
$$\begin{array}{lcl}
&=& h(e_i \otimes_{f^2_{1,i}} e_{i-1}) \beta_{i^0} \beta_{i^1} \\
&&- \alpha_{(i,0)} \alpha_{(i,1)} \cdots \alpha_{(i,q)} \beta_{i^0} \beta_{i^1} \cdots \beta_{i^{(p-1)}} h(e_{i^p} \otimes_{f^2_{2,i}} e_{i^1}) \beta_{i^1} \\
&&+ \beta_{(i+1)^0} h(e_{(i+1)^1} \otimes_{f^2_{4,i+1,1}} e_{i^2})\\
&=& d_i \alpha_{(i,0)} \alpha_{(i,1)} \cdots \alpha_{(i,q)} \beta_{i^0} \beta_{i^1} \cdots \beta_{i^p}\beta_{i^0} \beta_{i^1}
+ d_i'\alpha_{(i,0)} \alpha_{(i,1)} \cdots \alpha_{(i,q)} \beta_{i^0} \beta_{i^1} \\
&=& d_i'\alpha_{(i,0)} \alpha_{(i,1)} \cdots \alpha_{(i,q)} \beta_{i^0} \beta_{i^1}.
\end{array}$$
As $h \in {\Ker}\,d_3$ we have $d_i' =0$ for $i \in \{2, \ldots, k\}$.

Similarly it can be shown that $hA_3(e_1\otimes_{f^3_{2,1}} e_{1^2}) =  d_1'\alpha_{(1,0)} \alpha_{(1,1)} \cdots \alpha_{(1,q)} \beta_{1^0} \beta_{1^1}$ so that $d_1' =0$.

We also have $hA_3(\mo(f^3_{3,i,t'}) \otimes_{f^3_{3,i,t'}} \mt(f^3_{3,i,t'})) = 0$ for $i \in \{1, \ldots, k\}$ and $t' \in \{1, \ldots, q\}.$ Finally, putting
$hA_3(\mo(f^3_{4,i,j}) \otimes_{f^3_{4,i,j}} \mt(f^3_{4,i,j})) = 0$ does not give any new information for $i \in \{1, \ldots, k\}, j \in \{1, \ldots, p\}$.

\vspace*{.3cm}
Thus $h$ is given by
$$\begin{array}{rcl}
\mo(f^2_{1,i}) \otimes \mt(f^2_{1,i}) & \mapsto & d_i \alpha_{(i,0)} \alpha_{(i,1)} \cdots \alpha_{(i,q)} \beta_{i^0} \beta_{i^1} \cdots \beta_{i^p} \mbox{   for } i \in \{1,2, \ldots, k\},\\
\mo(f^2_{3,i}) \otimes \mt(f^2_{3,i}) & \mapsto & d_i''  \alpha_{(i,q)} \beta_{i^0} \beta_{i^1} \cdots \beta_{i^p} \alpha_{(i-1,0)} \mbox{   for } i \in \{1, \ldots, k\},\\
\mbox{else } & \mapsto & 0,\\
\end{array}$$ where $d_i, d_i''$ for $i \in \{1,\ldots, k\}$ are in $K$. It is clear that there is no dependency between $d_i, d_i''$, and therefore ${\dim}\,{\Ker}\,d_3 = 2k$.

\begin{prop} \label{pro3}
For $\L =\L(p,q, k, s,\lambda)$ and $s=k-1$, we have ${\dim}\,{\Ker}\,d_3 = 2k.$
\end{prop}

Using Propositions \ref{pro2}, \ref{pro4}, \ref{pro3} and Theorem \ref{thm1} we get the main result of this section.

\begin{thm} \label{thm2}
For $\L = \L(p,q, k, s,\lambda)$ where $p,q,s,k$
 are integers such that $p,q \geq 0, k \geq 2, 1 \leq s \leq k-1,\gcd(s,k) = 1, \gcd(s+2,k) = 1$ and $\lambda \in K\setminus\{0\}$, we have ${\dim}\,{\HH^2(\L)} = 1.$
\end{thm}

We conclude this section by giving a deformation of $\L$ which arises from the non-zero element $h_1 + {\Im}\,d_2$ in ${\HH}^2(\L)$.
\vspace*{.5cm}

Let $\eta = h_1 + {\Im}\,d_2$. Recall that $ \rho_1 = \alpha_{(1,0)} \alpha_{(1,1)} \cdots \alpha_{(1,q)} \beta_{1^0} \beta_{1^1} \cdots \beta_{1^p}$. We introduce a new parameter $t$ and define
the algebra $\L_{\eta}$ to be the algebra $K{\mathcal Q}/I_{\eta}$
where $I_{\eta}$ is the ideal generated by the following elements:\\
(1) $f^2_{1,1} - t \rho_1, f^2_{1,j}$ where $j \in \{2, \ldots, k\},$\\
(2) for all $i \in \{1, \ldots, k\}$, $f^2_{2,i}, f^2_{3,i}, f^2_{4,i,j}, f^2_{5,i,t'},$ where $j \in \{1, \ldots, p-1\},$ $t' \in \{1, \ldots, q-1\},$\\
(3) $\rho_1 a$ for all arrows $a$ with $\mt(\rho_1) = \mo(a)$,\\
(4) $a\rho_1$ for all arrows $a$ with $\mt(a) = \mo(\rho_1).$

We now need to show that ${\dim}\,\L_{\eta} = {\dim}\,\L$ to verify that $\L_{\eta}$ is indeed a deformation
of $\L$. First of all, it is clear that ${\dim}\,e_j \L_{\eta} = {\dim}\,e_j \L$
for all $t$ and for all vertices $e_j$ with $e_j \neq e_1$. Now we consider $e_1 \L$ and $e_1 \L_{\eta}$ with $t \neq 1$, and
$e_1 \L_{\eta}$ with $t = 1$. These projective modules are described as
follows:
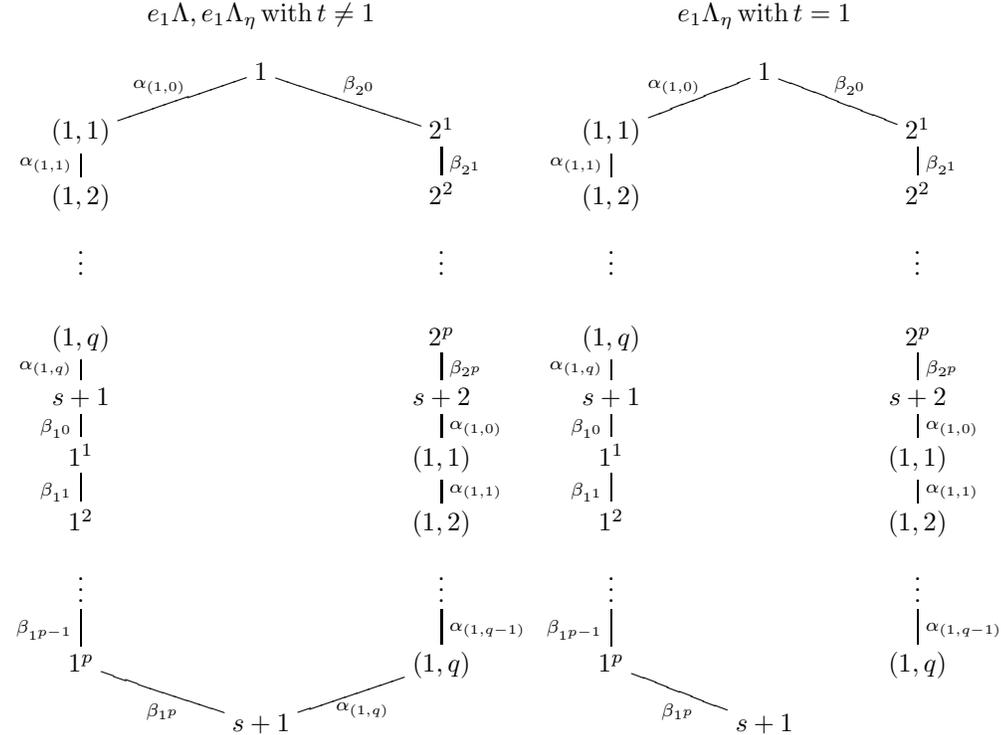
\vspace*{.3cm}
$$\xymatrix@R=8pt@C=8pt{
&e_1 \L \mbox{,\,} e_1 \L_{\eta} \mbox{\,with\,} t \neq 1 & &&&  & e_1 \L_{\eta} \mbox{\,with\,} t = 1& \\
& 1 \ar@{-}[dl]_{\alpha_{(1,0)}} \ar@{-}[dr]^{\beta_{2^0}} & & & & & 1 \ar@{-}[dl]_{\alpha_{(1,0)}} \ar@{-}[dr]^{\beta_{2^0}} & \\
(1,1) \ar@{-}[d]_{\alpha_{(1,1)}} & & 2^1 \ar@{-}[d]^{\beta_{2^1}} & & & (1,1) \ar@{-}[d]_{\alpha_{(1,1)}} & & 2^1 \ar@{-}[d]^{\beta_{2^1}} \\
(1,2) & & 2^2 & & & (1,2) & & 2^2 \\
\vdots & & \vdots & & & \vdots & & \vdots \\
(1,q) \ar@{-}[d]_{\alpha_{(1,q)}} & & 2^p \ar@{-}[d]^{\beta_{2^p}} & & & (1,q) \ar@{-}[d]_{\alpha_{(1,q)}} & & 2^p \ar@{-}[d]^{\beta_{2^p}} \\
s+1 \ar@{-}[d]_{\beta_{1^0}} & & s+2 \ar@{-}[d]^{\alpha_{(1,0)}}  & & & s+1 \ar@{-}[d]_{\beta_{1^0}} & & s+2 \ar@{-}[d]^{\alpha_{(1,0)}} \\
1^1 \ar@{-}[d]_{\beta_{1^1}} & & (1,1) \ar@{-}[d]^{\alpha_{(1,1)}} & & & 1^1 \ar@{-}[d]_{\beta_{1^1}} & & (1,1) \ar@{-}[d]^{\alpha_{(1,1)}}\\
1^2 & & (1,2) & & & 1^2 & & (1,2) \\
\vdots \ar@{-}[d]_{\beta_{1^{p-1}}} & & \vdots\ar@{-}[d]^{\alpha_{(1,q-1)}} & & & \vdots\ar@{-}[d]_{\beta_{1^{p-1}}}  &&\vdots\ar@{-}[d]^{\alpha_{(1,q-1)}}\\
1^p \ar@{-}[dr]_{\beta_{1^p}} & & (1,q) \ar@{-}[dl]^{\alpha_{(1,q)}} & & & 1^p \ar@{-}[dr]_{\beta_{1^p}} & & (1,q)  \\
& s+1 &  & & & & s+1 & \\
}$$

\vspace*{.5cm}
In each case we see that ${\dim}\,e_1 \L = {\dim}\,e_1 \L_{\eta} = 2p + 2q + 4$ for all $t$. Hence ${\dim}\,\L_{\eta} = {\dim}\,\L.$
Moreover, when $t=1$ the algebras $\L$ and $\L_\eta$ are not isomorphic
since, in this case, $\L_\eta$ is not self-injective. Thus we have found a non-trivial deformation of $\L$.

\begin{thm} \label{thm3}
With $\L, \eta,$ and $\L_{\eta}$ as defined above, then $\L_{\eta}$ is a
non-trivial deformation of $\L$. Moreover, the algebras $\L$ and $\L_\eta$ are socle equivalent.
\end{thm}

\bigskip

\section{${\HH}^2(\L)$ for $\L = \Gamma^*(n)$} \label{sec4}
We have given the algebra $\L = \Gamma^*(n)$ by quiver and relations in Section \ref{sec1}. Note that these
relations are not minimal. So we will find a minimal set of relations $f^2$ for this algebra.

Let
$$f^2_{1,1}  = \alpha_1 \alpha_2 - \gamma_1 \gamma_2, \hspace{1cm} f^2_{1,2}= \alpha_1 \alpha_2 - (\beta_1 \beta_2 \cdots \beta_n)^2$$

$$ f^2_{2,1} = \beta_n\alpha_1,  \hspace{1cm}
f^2_{2,2} = \beta_n \gamma_1,$$

$$f^2_{2,3} = \alpha_2 \beta_1,  \hspace{1cm}
f^2_{2,4} = \gamma_2 \beta_1,$$

$$ f^2_{2,5} = \alpha_2 \alpha_1,  \hspace{1cm}
f^2_{2,6} = \gamma_2 \gamma_1,$$

$$f^2_{3,j} = \beta_j \beta_{j+1} \cdots  \beta_n \beta_1 \cdots \beta_n \beta_1 \cdots \beta_{j-1}\beta_j, \mbox{ for } j\in \{2, \ldots, n-1\}.$$

The remaining relation $\beta_n (\beta_1 \beta_2 \cdots \beta^n)^2$ can be written as $f^2_{2,1} \alpha_2 - \beta_n f^2_{1,2}$. So this relation is in $I$ and is not in $f^2$.

\begin{prop}\label{proposition1}
For $\L = \Gamma^*(n)$ and with the above notation, the minimal set of relations is
$f^2 = \{f_{1,1}^2, f_{1,2}^2, f_{2,1}^2, f_{2,2}^2, f_{2,3}^2,
f_{2,4}^2, f_{2,5}^2, f_{2,6}^2, f_{3,j}^2 \mbox{ for } j = 2, \ldots, n -1\}$.
\end{prop}

Recall that the projective $Q^3 = \bigoplus_{y \in f^3} \L \mo(y) \otimes
\mt(y) \L$. Thus we have $Q^3 = (\L e_1 \otimes e_2\L) \oplus (\L  e_1 \otimes e_3\L) \oplus (\L e_1 \otimes e_{n+1}\L) \oplus (\L e_1 \otimes e_{n+2} \L)
\oplus (\L e_{n+1} \otimes e_1\L) \oplus (\L e_{n+2} \otimes e_1 \L) \oplus (\L
e_{n-1} \otimes e_1\L) \oplus \bigoplus_{m = 2}^{n - 2} (\L e_m \otimes e_{m+2}\L).$
(We note that the projective $Q^3$ is also described in \cite{H} although
Happel gives no description of the maps in the $\L, \L$-projective
resolution of $\L$.) Following \cite{GS}, and with the notation introduced
in Section \ref{sec2}, we may choose the set $f^3$ to consist of the
following elements:
$$\{f^3_{1,1}, f^3_{1,2}, f^3_{1,3}, f^3_{1,4}, f^3_{n+1}, f^3_{n+2}, f^3_{n-1},f^3_m\}, \mbox{ with } m \in \{2, \ldots, n - 2\} \mbox{  where }$$
$$\begin{array}{l c l c l l l l}
f^3_{1,1} & = & f^2_{1,1} \beta_1 \\
& = & \alpha_1 f^2_{2,3} - \gamma_1 f^2_{2,4}
 & \in & e_1 K{\mathcal Q} e_2,\\
f^3_{1,2} &=& f^2_{1,2} \beta_1 \beta_2 \\
 &=& \alpha_1 f^2_{2,3}\beta_2 -  \beta_1 f^2_{3,2}
 & \in & e_1 K{\mathcal Q} e_3,\\
f^3_{1,3}& = &f^2_{1,2} \alpha_1 \\
&=& \alpha_1 f^2_{2,5} - \beta_1\beta_2 \cdots \beta_n \beta_1\cdots \beta_{n-1} f^2_{2,1}
 & \in & e_1 K{\mathcal Q} e_{n+1},\\
f^3_{1,4} &=& f^2_{1,2} \gamma_1 - f^2_{1,1} \gamma_1 \\
&=& \gamma_1 f^2_{2,6} - \beta_1\beta_2 \cdots \beta_n \beta_1 \cdots \beta_{n-1} f^2_{2,2}
& \in & e_1 K{\mathcal Q} e_{n+2},\\
f^3_{n+1} &=& f^2_{2,5} \alpha_2 - f^2_{2,3} \beta_2 \cdots \beta_n \beta_1 \cdots \beta_n \\
&=&  \alpha_2 f^2_{1,2}
 &  \in & e_{n+1} K{\mathcal Q} e_1,\\
 \end{array}$$
 $$\begin{array}{l c l c l l l l}
f^3_{n+2} &=& f^2_{2,4} \beta_2 \cdots \beta_n \beta_1 \cdots \beta_n - f^2_{2,6} \gamma_2 \\
&=& \gamma_2 f^2_{1,1} - \gamma_2 f^2_{1,2}
 &  \in & e_{n+2}K{\mathcal Q} e_1,\\
 f^3_{n-1} &=& f^2_{3,n-1} \beta_n \\
&=& \beta_{n-1} f^2_{2,1} \alpha_2 - \beta_{n-1} \beta_n f^2_{1,2}
 &  \in & e_{n-1} K{\mathcal Q} e_1,\\
f^3_m &=& f^2_{3,m} \beta_{m+1} \\
&=& \beta_m f^2_{3,m+1}
 &  \in & e_m K{\mathcal Q} e_{m+2}\\
& & & & \mbox{ for } m \in \{2, \ldots, n - 2\}. \\
\end{array}$$

We know that ${\HH}^2(\L) = {\Ker}\,d_3 / {\Im}\,d_2$. First we will find
${\Im}\,d_2$. Let $f \in {\Hom}(Q^1, \L)$ and so write
$$f(e_1 \otimes_{\alpha_1} e_{n+1}) = c_1 \alpha_1, \hspace{1cm} f(e_{n+1}
\otimes_{\alpha_2} e_1) = c_2 \alpha_2,$$
$$f(e_1 \otimes_{\gamma_1} e_{n+2}) = c_3 \gamma_1, \hspace{1cm} f(e_{n+2}
\otimes_{\gamma_2} e_1) = c_4 \gamma_2,$$
$$f(e_k \otimes_{\beta_k} e_{k+1}) = d_k \beta_k + d'_k \beta_k \cdots \beta_n \beta_1 \cdots \beta_{k-1} \beta_k, \mbox{ for } k
\in \{1, \ldots, n\},$$
where $c_1, c_2, c_3, c_4, d_k,d'_k \in K \mbox{ for } k \in \{1, \ldots, n\}.$

Now we find $fA_2 = d_2f$. We have
$fA_2(e_1 \otimes_{f^2_{1,1}} e_{1}) = f(e_1 \otimes_{\alpha_1} e_{n+1})\alpha_2
+ \alpha_1 f(e_{n+ 1}\otimes_{\alpha_2} e_{1}) - f(e_{1} \otimes_{\gamma_1} e_{n+2})\gamma_2 - \gamma_1 f(e_{n+2} \otimes_{\gamma_2} e_{1}) =
c_{1} \alpha_1 \alpha_2 + c_2 \alpha_1 \alpha_2 - c_{3} \gamma_1\gamma_2 -  c_{4}\gamma_1\gamma_2
=(c_{1} + c_{2} - c_{3}- c_{4})\alpha_1 \alpha_2.$

Also
$fA_2(e_{1} \otimes_{f^2_{1,2}} e_{1}) =  f(e_{1} \otimes_{\alpha_1}
e_{n+1})\alpha_2 + \alpha_1 f(e_{n+1} \otimes_{\alpha_2} e_{1}) - f(e_{1}
\otimes_{\beta_{1}} e_{2}) \beta_2 \cdots \beta_n$ $\beta_1 \cdots \beta_n -
\ldots - \beta_{1} \cdots \beta_{n-1} f(e_{n} \otimes_{\beta_n} e_{1}) \beta_1 \cdots \beta_n
- \beta_1 \cdots \beta_n f(e_1 \otimes_{\beta_1} e_2) \beta_2 \cdots \beta_n
 - \ldots
- \beta_1 \cdots \beta_n \beta_1  \cdots \beta_{n-1} f(e_n \otimes_{\beta_n} e_1)
= c_{1} \alpha_1 \alpha_2 + c_{2} \alpha_1 \alpha_2
-d_1 \beta_1 \cdots \beta_n \beta_1 \cdots \beta_n
- \ldots
- d_n \beta_1 \cdots \beta_n \beta_1 \cdots \beta_n
-d_1 \beta_1 \cdots \beta_n \beta_1 \cdots \beta_n
-\ldots
-d_n \beta_1 \cdots \beta_n \beta_1 \cdots \beta_n
= (c_{1} + c_{2}) \alpha_1 \alpha_2 - (2d_1 + \ldots + 2d_n) (\beta_1 \cdots \beta_n)^2
=(c_1 + c_2 - 2d_1 - \ldots - 2d_n)  \alpha_1 \alpha_2.$

We can show by direct calculation that $fA_2(\mo(f^2_j) \otimes \mathfrak{t}(f^2_j)) = 0$
for all $f^2_j \neq f^2_{1,1}, f^2_{1,2}$.

Thus $fA_2$ is given by
$$fA_2(e_1 \otimes_{f^2_{1,1}} e_{1}) = (c_{1} + c_{2} - c_{3}- c_{4})\alpha_1 \alpha_2 = c' \alpha_1 \alpha_2,$$
$$fA_2(e_{1} \otimes_{f^2_{1,2}} e_{1}) = (c_1 + c_2 - 2d_1 - \ldots - 2d_n) \alpha_1 \alpha_2 = c'' \alpha_1 \alpha_2.$$
So $\dim\,{\Im}\,d_2 = 2$.

\begin{prop} \label{proposition2}
For $\L = \Gamma^*(n)$, we have ${\dim}\,{\Im}\,d_2 = 2.$
\end{prop}

Now we determine ${\Ker}\,d_3$. Let $h \in {\Ker}\,d_3$, so  $h \in
{\Hom}(Q^2, \L)$ and $d_3h = 0$. Then $h: Q^2 \rightarrow \L$ is given by

$$h(e_1 \otimes_{f_{1,1}^2} e_1) = c_1 e_1 + c_2 \alpha_1 \alpha_2 + c_ 3 \beta_1 \beta_2 \cdots \beta_n,$$
$$h(e_1 \otimes_{f_{1,2}^2} e_1) = c_4 e_1 + c_5 \alpha_1 \alpha_2 + c_6 \beta_1 \beta_2 \cdots \beta_n,$$
$$h(\mo(f^2_{2,l}) \otimes_{f^2_{2,l}} \mt(f^2_{2,l})) = 0, \mbox{ for } l
\in \{1, \ldots, 4\},$$
$$h(e_{n + 1} \otimes_{f_{2,5}^2} e_{n+1}) = c_7 e_{n+1},$$
$$h(e_{n+2} \otimes_{f_{2,6}^2} e_{n+2}) = c_8 e_{n+2} \mbox{ and }$$
$$h(\mo(f^2_{3,j}) \otimes_{f^2_{3,j}} \mt(f^2_{3,j})) = d_j \beta_j + d'_j \beta_j \beta_{j+1} \cdots \beta_n \beta_1 \cdots \beta_j,
\mbox{ for } j \in \{2, \ldots, n - 1\}$$
for some $c_1, \ldots, c_8, d_j, d'_j \in K$ for $j \in \{2, \ldots, n-1\}.$

Then $hA_3(e_1 \otimes_{f^3_{1,1}} e_{2}) = h(e_1 \otimes_{f^2_{1,1}}
e_1) \beta_1 - \alpha_1 h(e_{n + 1}\otimes_{f^2_{2,3}}
e_{2}) + \gamma_1 h(e_{n+2}
\otimes_{f^2_{2,4}} e_{2}) = (c_1 e_1 + c_2 \alpha_1 \alpha_2 + c_3 \beta_1 \beta_2 \cdots \beta_n)\beta_1 -0 + 0
= c_1 \beta_1 + c_3 \beta_1 \beta_2 \cdots \beta_n \beta_1.$ As $h \in {\Ker}\,d_3$ we have $c_1=0$ and $c_3 = 0.$

$hA_3(e_1 \otimes_{f^3_{1,2}} e_{3}) = h(e_1 \otimes_{f^2_{1,2}}
e_1) \beta_1 \beta_2 - \alpha_1 h(e_{n + 1}\otimes_{f^2_{2,3}}
e_{2}) \beta_2 + \beta_1 h(e_{2}
\otimes_{f^2_{3,2}} e_{3}) = (c_4 e_1 + c_5 \alpha_1 \alpha_2 + c_6 \beta_1 \beta_2 \cdots \beta_n)\beta_1 \beta_2 -0 + \beta_1(d_2 \beta_2 + d'_2 \beta_2 \cdots \beta_n \beta_1 \beta_2)= (c_4 + d_2) \beta_1 \beta_2+ (c_6 + d'_2) \beta_1 \beta_2 \cdots \beta_n \beta_1 \beta_2.$ As $h \in {\Ker}\,d_3$ we have $c_4 + d_2=0$ and $c_6+ d'_2 = 0.$ So $d_2 = -c_4$ and $d'_2 = -c_6.$

Next, $hA_3(e_1 \otimes_{f^3_{1,3}} e_{n+1}) = h(e_1 \otimes_{f^2_{1,2}}
e_1) \alpha_1 - \alpha_1 h(e_{n + 1}\otimes_{f^2_{2,5}}
e_{n+1}) + \beta_1 \beta_2 \cdots \beta_n \beta_1 \cdots$ $ \beta_{n-1} h(e_{n}
\otimes_{f^2_{2,1}} e_{n+1}) = (c_4 e_1 + c_5 \alpha_1 \alpha_2 + c_6 \beta_1 \beta_2 \cdots \beta_n)\alpha_1 - c_7 \alpha_1 +0
= (c_4 -c_7)  \alpha_1.$ So we have $c_4 - c_7 = 0$ and hence $c_7 = c_4.$

$hA_3(e_1 \otimes_{f^3_{1,4}} e_{n+2}) = h(e_1 \otimes_{f^2_{1,2}}
e_1) \gamma_1 -  h(e_{1}\otimes_{f^2_{1,1}}
e_{1}) \gamma_1 - \gamma_1 h(e_{n+2}\otimes_{f^2_{2,6}} e_{n+2}) + \beta_1 \beta_2 \cdots \beta_n\beta_1  \cdots \beta_{n-1} h(e_{n} \otimes_{f^2_{2,2}} e_{n+2})
= (c_4 e_1 + c_5 \alpha_1 \alpha_2 + c_6 \beta_1 \beta_2 \cdots \beta_n)\gamma_1 - (c_1 e_1 + c_2 \alpha_1 \alpha_2 + c_3 \beta_1 \beta_2 \cdots \beta_n)\gamma_1 -c_8 \gamma_1 + 0 = (c_4 - c_1 - c_8) \gamma_1.$ Therefore $c_8 = c_4$ as $c_1 = 0.$

$hA_3(e_{n+1} \otimes_{f^3_{n+1}} e_{1}) = h(e_{n+1} \otimes_{f^2_{2,5}}
e_{n+1}) \alpha_2 - h(e_{n + 1}\otimes_{f^2_{2,3}} e_{2}) \beta_2 \cdots \beta_n \beta_1 \cdots \beta_n -
\alpha_2 h(e_{1} \otimes_{f^2_{1,2}} e_{1})
= c_7 \alpha_2  - 0 - \alpha_2 (c_4 e_1 + c_5 \alpha_1 \alpha_2 + c_6 \beta_1 \beta_2 \cdots \beta_n)=(c_7 - c_4) \alpha_2.$ Thus again we have $c_7 = c_4.$

$hA_3(e_{n+2} \otimes_{f^3_{n+2}} e_{1}) =  h(e_{n + 2}\otimes_{f^2_{2,4}} e_{2}) \beta_2 \cdots \beta_n \beta_1 \cdots \beta_n
- h(e_{n+2} \otimes_{f^2_{2,6}} e_{n+2}) \gamma_2 - \gamma_2 h(e_{1} \otimes_{f^2_{1,1}} e_{1}) + \gamma_2 h(e_1 \otimes_{f^2_{1,2}} e_1)
= 0 - c_8 \gamma_2 - \gamma_2 (c_1 e_1 + c_2 \alpha_1 \alpha_2 + c_3 \beta_1 \beta_2 \cdots \beta_n) + \gamma_2 (c_4 e_1 + c_5 \alpha_1 \alpha_2 + c_6 \beta_1 \beta_2 \cdots \beta_n) = (-c_8 - c_1 + c_4) \gamma_2.$ As $c_1 = 0$ above, we have $c_8 = c_4$ as we already know.

Also $hA_3(e_{n-1} \otimes_{f^3_{n-1}} e_{1}) = h(e_{n-1} \otimes_{f^2_{3,n-1}} e_{n}) \beta_n - \beta_{n-1} h(e_{n} \otimes_{f^2_{2,1}} e_{n+1}) \alpha_2 + \beta_{n-1} \beta_n h(e_1 \otimes_{f^2_{1,2}} e_1)
= (d_{n-1} \beta_{n-1} + d'_{n-1} \beta_{n-1} \beta_n \beta_1 \cdots \beta_{n-1}) \beta_n + \beta_{n-1} \beta_n (c_4 e_1 + c_5 \alpha_1 \alpha_2 + c_6 \beta_1 \beta_2 \cdots \beta_n)
= d_{n-1} \beta_{n-1} \beta_n + d'_{n-1} \beta_{n-1} \beta_n \beta_1 \cdots \beta_{n-1} \beta_n + c_4 \beta_{n-1} \beta_n + c_6 \beta_{n-1} \beta_n \beta_1 \beta_2 \cdots \beta_n
=(d_{n-1} + c_4) \beta_{n-1} \beta_n + (d'_{n-1} + c_6) \beta_{n-1} \beta_n \beta_1 \cdots \beta_{n-1} \beta_n.$ So we have $d_{n-1} = -c_4$ and $d'_{n-1} = -c_6.$

Finally, for $2 \leq m \leq n-2$, we have $hA_3(e_m \otimes_{f^3_{m}} e_{m+2}) = h(e_m \otimes_{f^2_{3,m}} e_{m+1}) \beta_{m+1} - \beta_m h(e_{m+1} \otimes_{f^2_{3,m+1}} e_{m+2})$
$=(d_m \beta_m + d'_m \beta_m \beta_{m+1} \cdots \beta_n \beta_1 \cdots \beta_m) \beta_{m+1} - \beta_m (d_{m+1}$ $\beta_{m+1} + d'_{m+1} \beta_{m+1} \beta_{m+2} \cdots \beta_n \beta_1 \cdots \beta_{m+1})$
$= (d_m - d_{m+1}) \beta_m \beta_{m+1} + (d'_m - d'_{m+1}) \beta_m$ $\beta_{m+1} \cdots \beta_n \beta_1 \cdots \beta_m \beta_{m+1}.$ Therefore we have $d_m = d_{m+1}$ and $d'_m = d'_{m+1}$. Hence $d_m = -c_4$ and $d'_m = -c_6$ for $m \in \{2, \ldots, n-1\}$ as we have above $d_2 = d_{n-1} = -c_4$ and $d'_2 = d'_{n-1} = -c_6.$

Thus $h$ is given by
$$h(e_1 \otimes_{f_{1,1}^2} e_1) = c_2 \alpha_1 \alpha_2,$$
$$h(e_1 \otimes_{f_{1,2}^2} e_1) = c_4 e_1 + c_5 \alpha_1 \alpha_2 + c_6 \beta_1 \beta_2 \cdots \beta_n,$$
$$h(\mo(f^2_{2,l}) \otimes_{f^2_{2,l}} \mt(f^2_{2,l})) = 0, \mbox{ for } l
\in \{1, \ldots, 4\},$$
$$h(e_{n + 1} \otimes_{f_{2,5}^2} e_{n+1}) = c_4 e_{n+1},$$
$$h(e_{n+2} \otimes_{f_{2,6}^2} e_{n+2}) = c_4 e_{n+2} \mbox{ and }$$
$$h(\mo(f^2_{3,j}) \otimes_{f^2_{3,j}} \mt(f^2_{3,j})) = -c_4 \beta_j -c_6 \beta_j \beta_{j+1} \cdots \beta_n \beta_1 \cdots \beta_j,
\mbox{ for } j \in \{2, \ldots, n - 1\}$$
for some $c_2, c_4, c_5, c_6 \in K.$

\begin{prop} \label{proposition3}
For $\L = \Gamma^*(n)$, we have ${\dim}\,{\Ker}\,d_3 = 4.$
\end{prop}

Therefore  $\dim\,{\HH}^2(\L) = \dim\,{\Ker}\,d_3 - \dim\,{\Im}\,d_2 = 4 - 2= 2$ and a basis is given by the maps $\eta_1$ and $\eta_2$ where
$\eta_1$ is given by
$$\begin{array}{rcl}
e_1 \otimes_{f_{1,2}^2} e_1 & \mapsto & e_1,\\
e_{n+1} \otimes_{f^2_{2,5}} e_{n+1} & \mapsto & e_{n+1},\\
e_{n+2} \otimes_{f^2_{2,6}} e_{n+2} & \mapsto & e_{n+2},\\
\mo(f^2_{3,j}) \otimes_{f^2_{3,j}} \mt(f^2_{3,j}) & \mapsto & - \beta_j,
\mbox{ for } j \in \{2, \ldots, n - 1\},\\
\mbox{else } & \mapsto & 0,
\end{array}$$

$\eta_2$ is given by
$$\begin{array}{rcl}
e_1 \otimes_{f_{1,2}^2} e_1 & \mapsto & \beta_1 \beta_2 \cdots \beta_n,\\
\mo(f^2_{3,j}) \otimes_{f^2_{3,j}} \mt(f^2_{3,j}) & \mapsto & - \beta_j \beta_{j+1} \cdots \beta_n \beta_1 \cdots \beta_j,
\mbox{ for } j \in \{2, \ldots, n - 1\},\\
\mbox{else } & \mapsto & 0.
\end{array}$$

From Proposition \ref{proposition2} and Proposition \ref{proposition3} we get the main result of this section.
\begin{thm}\label{theorem1}
For $\L = \Gamma^*(n)$ with $n \geq 1$ we have $\dim\,{\HH}^2(\L) = 2.$
\end{thm}

To connect this with deformations we use a similar discussion as Section \ref{sec3}. We introduce the parameter $t$ and define the algebra $\L_{\eta_2}$ to be the algebra $K{\mathcal Q}/I_{\eta_2}$
where $I_{\eta_2}$ is the ideal generated by the following elements:\\
(1) $f^2_{1,1},\\
(2) f^2_{1,2} - t \beta_1 \beta_2 \cdots \beta_n,\\
(3) f^2_{2,1}, f^2_{2,2}, f^2_{2,3}, f^2_{2,4}, f^2_{2,5}, f^2_{2,6},\\
(4) f^2_{3,j} + t \beta_j \beta_{j+1} \cdots  \beta_n \beta_1 \cdots \beta_{j-1}\beta_j, \mbox{ for } j\in \{2, \ldots, n-1\}.$

We can show that ${\dim}\,\L_{\eta_2} \neq {\dim}\,\L$. Hence this algebra has no non-trivial deformation.

From Theorem \ref{thm2} and Theorem \ref{theorem1} we have now found ${\HH}^2(\L)$ for all standard one-parametric but not weakly symmetric self-injective algebras of tame representation type.

\section*{Acknowledgements}
I thank Prof. Nicole Snashall for her encouragement and helpful comments.

\end{document}